\documentclass[12pt]{amsart}

\usepackage{amssymb}
\usepackage{pb-diagram}
\usepackage{enumitem} 
\usepackage[dvipsnames]{xcolor}
\usepackage{tikz-cd}
\usepackage{graphicx}
\usepackage{amsmath}
\numberwithin{equation}{section}
\usepackage[all]{xy}

\newtheorem{Theorem}{Theorem}[section]
\newtheorem{Proposition}[Theorem]{Proposition}
\newtheorem{Lemma}[Theorem]{Lemma}
\newtheorem{Corollary}[Theorem]{Corollary}
\newtheorem{Definition}[Theorem]{Definition}
\newtheorem{Remark}[Theorem]{Remark}
\newtheorem{Example}[Theorem]{Example}

\newtheorem{Notation}[Theorem]{Notation}

\numberwithin{equation}{section}
\newcommand{\rank}{\operatorname{rank}}
\newcommand {\ZZ}{\mathbb{Z}}

\begin{document}

\title[Clifford's Theorem for Coherent Systems]{Clifford's Theorem for Coherent Systems on surfaces}

\author{L. Costa}
\address{Facultat de Matem\`atiques i Inform\`atica,
Departament de Matem\`atiques i Inform\'atica, Gran Via 
585, 08007 Barcelona, SPAIN } \email{costa@ub.edu}
\thanks{$^*$ Partially supported by PID2020-113674GB-I00.}

\author{I. Mac\'ias Tarr\'io}
\address{Facultat de Matem\`atiques i Inform\`atica,
Departament de Matem\`atiques i Inform\`atica, Gran Via
585, 08007 Barcelona, SPAIN} \email{irene.macias@ub.edu}
\thanks{$^{**}$ Partially supported by PID2020-113674GB-I00.}

\author{L. Roa-Leguizam\'on}

\address{Universidad Antonio Nariño, Departamento de matem\'aticas \newline Calle 58A Bis  37 - 94, Bogota, Colombia.}

\email{leonardo.roa@cimat.mx}



\date{\today}

\baselineskip=16pt

\begin{abstract}
Let $X$ be a smooth irreducible projective surface. The aim of this paper is to establish a version 
 of Clifford’s theorem for coherent systems on $X$. 
\end{abstract}

\maketitle

\section{Introduction}

Let $X$ be a smooth irreducible projective variety. A coherent system on $X$ is a pair $(\mathcal{E},\mathcal{V})$ where $\mathcal{E}$ is a coherent sheaf
 on $X$,  and
$\mathcal{V} \subseteq H^{0}(X,\mathcal{E})$. These objects were introduced in the 90's by Le Potier in \cite{Potier}. Associated to coherent systems there is a notion of stability which allows the construction of the moduli spaces. Many interesting results have been proved regarding these moduli spaces
when the underlying variety is a curve (see for instance \cite{Bradlow, Bradlow-Garcia-Prada, Bradlow-Garcia-Prada-Mercat-Munoz-Newstead, Bradlow-Garcia-Prada-Munoz-Newstead1}), but very little is known if the variety has
dimension greater than or equal to two.

There is a useful relation between
moduli spaces of stable coherent systems and the Brill-Noether loci inside  moduli spaces of semistable bundles with  fixed  rank and Chern classes. Roughly speaking, a $k$-Brill-Noether subvariety is a subvariety of the moduli space of semistable vector bundles whose points correspond to bundles having at least $k$ independent global sections.  The main goal of Brill-Noether theory is the study of these subvarieties. In fact, Brill-Noether theory focuses  in  questions concerning non-emptiness, connectedness, irreducibility, dimension, singularities, topological and geometric structures, among others, of theses subvarieties. For line bundles on curves (Classical Brill-Noether theory), many of these questions have been answered when the curve is generic  (see \cite{Arbarello-Cornalba-Griffiths-Harris}) but, in spite of the increasing interest in recent years,  much less is known about vector bundles of higher rank on curves, surfaces, 3-folds and on varieties of higher dimension (see for instance  \cite{G-T, Costa-Miro, ALH}, among others).

On the other hand, Clifford's theorem for semistable bundles $E$ on curves or on surfaces determines an upper bound for the number of independent sections of $E$ in terms of their invariants and of the variety. Hence, Clifford's theorem can be seen as a first step in order to determine when Brill-Noether loci are nonempty. Since coherent systems have played an important role in the development of the Brill-Noether theory, it is natural to look for some sort of Clifford's theorem for coherent systems. The first to consider this point were  Lange and Newstead in the context of coherent systems on curves (\cite{Lange-Newstead}). In fact,  
 Clifford's theorem for coherent systems on curves states that, for any $\alpha-$semistable coherent system $(E,V)$ of type $(n,d,k)$ where $n=\text{rank}(E)$ , $d=\deg(E)$ and  $k=\dim(V)$ on a smooth projective curve $C$ of genus $g \geq 2$, we have
 \[ k \leq \begin{cases}
     \frac{d}{2}+n, \,\,\,\,\,\,\,\,\,\,\,\,\,\,\,\,\,\,\,\,\,\,\,\,\,\,\,\,\,\,\,\,\text{if} \,\,\,\,\,\,\,\, 0 \leq d \leq 2gn ,\,\,\,\,\,\,\, \text{and}\\
     d+n(1-g) \,\,\,\,\,\,\,\,\,\,\,\,\,\,\,\,\,\text{if} \,\,\,\,\,\,\,\,  d \geq 2gn. \,\,\,\,\,\,\,
 \end{cases}\]
 \noindent (See \cite[Theorem 2.1]{Lange-Newstead}).

 The aim of this  paper is to determine  a Clifford's theorem for coherent systems on a smooth irreducible surface. More precisely we will prove the following result:

\noindent \textbf{Theorem}
Let $X$ be a smooth projective surface, denote by $K_X$ its canonical divisor and let $H$ be an ample divisor on $X$ such that $K_X\cdot H\leq0$.
    Let $(E,V)$ be an $\alpha-$semistable coherent system on $X$ for some $\alpha \in \mathbb{Q}[m]_{>0}$. Let $a$ be an integer such that
\[\deg(X)\cdot\max\left\{\frac{\rank(E)^2-1}{4},1 \right\} < \frac{\binom{a+2}{2}-a-1}{a}.\] If \[0\leq \frac{c_1(E) H}{\rank(E)}<aH^2+K_X H,\] 
then 
    \[\dim(V) \leq \rank(E)+a\cdot \frac{c_1(E) H}{2}.\]

Next we outline the structure of this paper. In Section 2, we recall the basic facts on coherent systems and we prove technical results about semistable coherent systems that we will need later.  In Section 3, we establish the Clifford's theorem for coherent systems on irreducible smooth surfaces and we end the section with some applications. Through all the paper we will work over an algebraically closed field of characteristic 0 and $X$ is a smooth irreducible projective surface.  

\section{Coherent Systems}

  We start the section with a brief summary  on coherent systems on surfaces (for further treatment of the subject see \cite{Potier, He}). Then, we will get some results concerning semistable coherent systems that will be key ingredients to obtain our main result.

\begin{Definition}  Let $X$ be a smooth irreducible projective surface. 
    \begin{enumerate}
    \item[(1)] A \it{coherent system of dimension $d$ on $X$} is a pair $(\mathcal{E},\mathcal{V})$ where $\mathcal{E}$ is a coherent sheaf
 of dimension $d$ on $X$,  and
$\mathcal{V} \subseteq H^{0}(X,\mathcal{E})$.
\item[(2)] A morphism of coherent systems $(\mathcal{E}_1,\mathcal{V}_1) \to (\mathcal{E}_2,\mathcal{V}_2)$ is a morphism of coherent sheaves $\phi: \mathcal{E}_1 \to \mathcal{E}_2$ such that $H^0(\phi)(\mathcal{V}_1) \subset \mathcal{V}_2$.
\item[(3)]  A coherent subsystem of $(\mathcal{E},\mathcal{V})$ is a pair $(\mathcal{F},\mathcal{W})$ 
where $0 \neq \mathcal{F} \subset \mathcal{E}$ is a proper subsheaf of 
$\mathcal{E}$ and $\mathcal{W} \subseteq \mathcal{V}\cap H^{0}(X,\mathcal{F})$.
\item[(4)] A quotient coherent system  of $(\mathcal{E},\mathcal{V})$ is a coherent system $(\mathcal{G},\mathcal{Z})$ together with a morphism 
 $\phi: (\mathcal{E},\mathcal{V}) \to (\mathcal{G},\mathcal{Z})$ of coherent systems such that the morphism $\phi: \mathcal{E} \to \mathcal{G}$ is surjective and $H^0(\phi)(\mathcal{V})=\mathcal{Z}$.
\end{enumerate}
\end{Definition}

\begin{Remark} \rm
In general, a subsystem $(\mathcal{F},\mathcal{W}) \subset (\mathcal{E},\mathcal{V})$ does not define a quotient system. However, whenever $\mathcal{W}=\mathcal{V}\cap H^0(X,\mathcal{F})$ we have a corresponding 
 quotient system $(\mathcal{G},\mathcal{Z}):= (\mathcal{E}/\mathcal{F},\mathcal{V}/\mathcal{W})$ which fits in the exact sequence 
\[0 \longrightarrow (\mathcal{F},\mathcal{W}) \longrightarrow (\mathcal{E},\mathcal{V}) \longrightarrow (\mathcal{G},\mathcal{Z}) \longrightarrow 0.\]
\end{Remark}

In this paper, we restrict our attention to coherent systems $(\mathcal{E},\mathcal{V})$ of dimension 2, that is, on pairs $(\mathcal{E},\mathcal{V})$  where $\mathcal{E}$ is a torsion-free sheaf. Let us denote by $(E,V)$ the coherent systems with $E$ a torsion-free sheaf.

\begin{Remark} \rm
 If $(E,V)$ is a coherent system with $E$ a torsion-free sheaf, then for any coherent subsystem $(F,W) \subset (E,V)$, $F$ is torsion-free.
\end{Remark}

\begin{Definition} 
A coherent system of type $(n, c_1,c_2, k)$ is a coherent system $(E,V)$ where $E$ is a  rank $n$ torsion-free sheaf with Chern classes
$c_i \in H^{2i}(X,\mathbb{Z})$, for $i=1,2$ and $V \subset H^0(X,E)$ is a subspace of dimension $k$.
%If there is no confusion, to simplify the notation, we will say that $(E,V)$ is a coherent system meaning that it has a fixed type $(n,c_1,c_2,k)$. 
\end{Definition}

\begin{Notation} \rm
    For simplicity of notation we will  denote the dimension of a vector space $V$ by the corresponding lowercase letter $v$ and given a coherent torsion-free sheaf $E$ we denote by  $n_E$ its rank. 
\end{Notation}

Associated to the coherent systems there is a notion of stability which depends on a parameter $\alpha \in \mathbb{Q}[m]$. To introduce it, it is convenient to fix some notation.

\begin{Notation} \rm
We denote by $\mathbb{Q}[m]$ the space of polynomials on $m$ with coefficients on  $\mathbb{Q}$. Given a couple of polynomials $p_1,p_2 \in \mathbb{Q}[m]$, we write $p_1 \leq p_2$ if and only if $p_1(m)-p_2(m) \leq 0$ for $m \gg 0$  and we say that $\alpha \in \mathbb{Q}[m]_{>0}$ if $\alpha>0$.  
    \end{Notation}

    \begin{Definition}
   Let   $\alpha \in \mathbb{Q}[m]_{>0}$ and $H$ an ample divisor  on $X$. Given a coherent system $(E,V)$ of type $(n,c_1,c_2,k)$ on $X$, we define its \it{reduced Hilbert polynomial} by
    \begin{eqnarray*}
p^{\alpha}_{H,(E,V)}(m)=\frac{k}{n}\cdot \alpha +\frac{P_{H,E}(m)}{n},
\end{eqnarray*}
where $P_{H,E}(m)$ denotes the Hilbert polynomial of $E$.
    \end{Definition}

    Notice that, since $X$ is a smooth projective surface, the Hilbert polynomial with respect to $H$ of a rank $n$ torsion-free sheaf $E$ on $X$ with Chern classes $c_1$ and $c_2$ can be written as follows:  
\[\begin{array}{ll} 
\frac{P_{H,E}(m)}{n} &=\frac{H^2 m^2}{2} + \left[ \frac{c_1 H}{n}- \frac{K_X H}{2} \right]m 
+ \frac{1}{n}\left(\frac{c_1^2-(c_1 K_X)}{2}-c_2 \right)+\chi(\mathcal{O}_X)\\ 
&=\frac{H^2 m^2}{2} + \left[ \frac{c_1 H}{n}- \frac{K_X H}{2} \right]m+ \frac{\chi(E)}{n}\\
\end{array} \]
where $K_X$ denotes the canonical divisor of $X$ and $\chi(F)$ the Euler Characteristic of a sheaf $F$.

\begin{Definition}
Let   $\alpha \in \mathbb{Q}[m]$ and $H$ an ample divisor  on $X$. We say that $(E,V)$ is $\alpha$-stable (resp. $\alpha$-semistable) if for any proper coherent subsystem $0 \neq (F,W)\subset (E,V)$
the following inequality holds;
\begin{equation} \label{stability}
    p^{\alpha}_{H,(\mathcal{F},W)}(m) < p^{\alpha}_{H,(E,V)}(m), \,\,\,\,\,  (resp. \leq).
\end{equation}
\end{Definition}

 It is not hard to check that $\alpha \in \mathbb{Q}[m]_{>0}$ is a necessary condition for the existence of $\alpha-$semistable coherent systems (see \cite[Lemma 1.3]{Newstead}). Moreover,  note that in order to check the stability of a coherent system $(E,V)$, it  is enough to verify the inequality (\ref{stability}) for proper coherent subsystems $(F,W) \subset (E,V)$ such that $W=V\cap H^0(X,F)$. In addition we will see that we can characterize the stability in terms of torsion-free quotients. To this end, we need to introduce some technical results. 
 
\begin{Lemma}\label{additive}
    Let $\alpha\in \mathbb{Q}[m]_{>0}$  and  $H$ be an ample divisor on $X$. Consider the following exact sequence of coherent systems
    \[0 \to (F,W) \to (E,V) \to (G,Z) \to 0.\]
    The following holds:
    \begin{enumerate}
        \item If $p^{\alpha}_{H,(F,W)}(m) \leq p^{\alpha}_{H,(E,V)}(m)$, then $p^{\alpha}_{H,(E,V)}(m) \leq p^{\alpha}_{H,(G,Z)}(m)$,
 \item        If $p^{\alpha}_{H,(F,W)}(m) \geq p^{\alpha}_{H,(E,V)}(m)$, then $p^{\alpha}_{H,(E,V)}(m) \geq p^{\alpha}_{H,(G,Z)}(m)$, and
 \item        If $p^{\alpha}_{H,(F,W)}(m) =p^{\alpha}_{H,(E,V)}(m)$, then $p^{\alpha}_{H,(E,V)}(m) = p^{\alpha}_{H,(G,Z)}(m)$,
    \end{enumerate}
\end{Lemma}
\begin{proof}  
We will prove $(1)$. The proof of $(2)$ and $(3)$ follows exactly in the same way. Observe that $p^{\alpha}_{H,(F,W)}(m) \leq p^{\alpha}_{H,(E,V)}(m)$ if and only if
\[ \frac{c_1(F) H}{n_F}m+ \frac{\chi(F)}{n_F}+\alpha \frac{w}{n_F} \leq 
         \frac{c_1(E)H}{n_E}m + \frac{\chi(E)}{n_E}+\alpha \frac{v}{n_E}.
     \]
     Using the fact that by aditivity on short exact sequences, $\chi(E)=\chi(F)+\chi(G)$, $n_E=n_F+n_G$, $c_1(E)=c_1(F)+c_1(G)$ and $v=w+z$, this inequality is equivalent to 
 \[ \frac{(c_1(E)-c_1(G)) H}{n_E-n_G}m+ \frac{\chi(E)-\chi(G)}{n_E-n_G}+\alpha \frac{v-z}{n_E-n_G} \leq 
         \frac{c_1(E)H}{n_E}m + \frac{\chi(E)}{n_E}+\alpha \frac{v}{n_E}.
     \]    
     Hence,
     \[ -n_E(c_1(Q) Hm+ \chi(G)+\alpha z)\leq 
         -n_G(c_1(E)Hm + \chi(E)+\alpha v)
     \]    
 which is equivalent to $p^{\alpha}_{H,(E,V)}(m) \leq p^{\alpha}_{H,(G,Z)}(m)$.   
   \end{proof}

Let us now recall the following well known fact (see for instance \cite[Lemma 1.1.17]{Okonek}).

\begin{Lemma}\label{Monomorphism}
    Let $H$ be an ample divisor on $X$. If $F \to F'$ is a monomorphism of torsion-free sheaves of the same rank, then $c_1(F)H \leq c_1(F')H$.
\end{Lemma}

Now we are ready to state the characterization of stability by means of torsion-free quotients.

\begin{Proposition} \label{torsionfree}
    Let $\alpha\in \mathbb{Q}[m]_{>0}$  and $H$ be an ample divisor on $X$. Let $(E, V)$ be a  coherent system on $X$. Then, the following statements are equivalent:
    \begin{enumerate}
        \item $(E,V)$ is $\alpha-$stable.
        \item $p^\alpha_{H, (F,W)}(m) < p^\alpha_{H, (E,V)}(m)$ for all coherent subsystem $(F,W) \subset (E,V)$ with $0 < rank \, F < rank \, E$ whose quotient $(G,Z)$ has $G$ torsion-free.
        \item $p^\alpha_{H, (G,Z)}(m) > p^\alpha_{H, (E,V)}(m)$ for all quotients $(G,Z)$ of $(E,V)$ with $G$  torsion-free  and $0 < rank \,G < rank \, E$. 
    \end{enumerate}
\end{Proposition}
\begin{proof}
By definition, $(1)$ implies $(2)$. Let us see that $(2)$ implies $(1)$.  Let $(F,W)$ be a coherent subsystem of $(E,V)$ with $W=V\cap H^0(X,F)$. We are going to see that 
$p^\alpha_{H,(F,W)} (m)< p^\alpha_{H,(E,V)}(m)$. To this end,  
let $(\mathcal{G},Z):=(E/F, V/W)$ be the corresponding quotient system. Define 
        \[F_E : = \ker( E \to (\mathcal{G})/ T(\mathcal{G}))\]
        where $T(.)$ denotes the torsion of a sheaf. Note that $F \to F_E$ is a monomorphism of torsion-free sheaves of the same rank $n_F$ hence, by Lemma \ref{Monomorphism}, we have $c_1(F) H \leq c_1(F_E)H$. Finally, since $W=V\cap H^0(X,F) \subseteq V\cap H^0(X,F_E) $,  define the coherent system $(F_E, W)$. It follows from $(2)$ that $p^\alpha_{H,(F_E,W)}(m) < p^\alpha_{H,(E,V)}(m)$. Let us see  that 
        \[p^\alpha_{H,(F,W)}(m) \leq p^\alpha_{H,(F_E,W)}(m).\]
If $c_1(F) H < c_1(F_E)H$, then
\[\begin{aligned}
                p^\alpha_{H,(F,W)}(m) &= \frac{w}{n_F}\alpha +\frac{H^2 m^2}{2} + \left[ \frac{c_1(F) H}{n_F}- \frac{K_XH}{2} \right]m+ \frac{\chi(F)}{n_F} \\
                &< \frac{w}{n_F}\alpha +\frac{H^2 m^2}{2} + \left[ \frac{c_1(F_E) H}{n_F}- \frac{K_X H}{2} \right]m+ \frac{\chi(F_E)}{n_F} \\
                & =p^\alpha_{H,(F_E,W)}(m).
            \end{aligned}\]
Finally, if $c_1(F) H =c_1(F_E)H$, since $F_E/F$ is a torsion sheaf, $c_2(F_E/F) \leq 0$ and from the exact  sequence 
            \begin{equation}\label{equation}
                0 \to F \to F_E \to F_E/F \to 0
            \end{equation}
     we obtain $c_2(F_E)=c_2(F)+c_2(F/F_E)< c_2(F)$ which implies that  $\chi(F)\leq \chi(F_E)$ and hence $p^\alpha_{H,(F,W)}(m) \leq p^\alpha_{H,(F_E,W)}(m)$. 
            
Let us now see that $(2)$ implies $(3)$.  Let $(G,Z)$ be a quotient coherent system of $(E,V)$ with $G$ torsion-free. Let $(F, W)$ be the corresponding coherent subsystem which fits the following exact sequence
\[0 \to (F, W) \to (E,V) \to (G,Z) \to 0.\]
By hypothesis $p^\alpha_{H,(F,W)}(m) < p^\alpha_{H,(E,V)}(m)$, and thus from Lemma \ref{additive} we conclude that $p^\alpha_{H,(E,V)}(m) < p^\alpha_{H,(G,Z)}(m)$. The converse follows exactly in the same way.         
\end{proof}

\begin{Definition}
    Let $\alpha\in\mathbb{Q}[m]_{>0}$. We say that $\alpha$ is a regular value if there exists $\beta_1, \beta_2 \in \mathbb{Q}[m]_{>0}$ with $\beta_1<\alpha < \beta_2$  such that $(E,V)$ is $\gamma$-stable  for any  $\gamma\in(\beta_1,\beta_2)$. If $\alpha$ is not a regular value we say that it is a critical value.
\end{Definition}

\begin{Proposition}
\label{prop_valorcritico}
    Let $\alpha\in\mathbb{Q}[m]_{>0}$ and  $H$ be an ample divisor on $X$. Let $(E,V)$ be an $\alpha$-semistable coherent system of type $(n_E,c_1,c_2,v)$. Then $\alpha$ is a critical value if and only if there exists a coherent subsystem $(E',V')\subset (E,V)$ with $\frac{v'}{n_{E'}}\neq\frac{v}{n_{E}}$ such that $p^\alpha_{H,(E',V')}(m)=p^\alpha_{H,(E,V)}(m)$.
\end{Proposition}

\begin{proof}

    Assume that there exists a coherent subsystem $(E',V')\subset (E,V)$ of type $(n_{E'},c_1,c_2,v')$, with $\frac{v'}{n_{E'}}\neq\frac{v}{n_{E}}$, such that $p^\alpha_{H,(E',V')}(m)=p^\alpha_{H,(E,V)}(m)$. Note that $p^\alpha_{H,(E',V')}(m)=p^\alpha_{H,(E,V)}(m)$ is equivalent to $\alpha=\frac{1}{\frac{v'}{n_{E'}}-\frac{v}{n_E}}\left(\frac{P_{H,E}(m)}{n_E}-\frac{P_{H,E'}(m)}{n_{E'}}\right)$.
    Let us  assume that $\frac{v'}{n_{E'}}-\frac{v}{n_E}>0$, and  consider $\beta>\alpha$. Notice that $$\beta>\alpha=\frac{1}{\frac{v'}{n_{E'}}-\frac{v}{n_E}}\left(\frac{P_{H,E}(m)}{n_E}-\frac{P_{H,E'}(m)}{n_{E'}}\right)$$ is equivalent to $$ \frac{P_{H,E'}(m)}{n_{E'}}+\beta \frac{v'}{n_{E'}}>\frac{P_{H,E}(m)}{n_E}+\beta \frac{v}{n_E}$$
    which implies that $(E,V)$ is $\beta$-unstable for $\beta >\alpha$. If  $\frac{v'}{n_{E'}}-\frac{v}{n_{E'}}<0$, we consider $$\beta<\alpha=\frac{1}{\frac{v'}{n_{E'}}-\frac{v}{n_E}}\left(\frac{P_{H,E}(m)}{n_E}-\frac{P_{H,E'}(m)}{n_{E'}}\right)$$ which is equivalent to $$ \frac{P_{H,E'}(m)}{n_{E'}}+\beta \frac{v'}{n_{E'}}>\frac{P_{H,E}(m)}{n_E}+\beta \frac{v}{n_E}$$
    which implies that $(E,V)$ is $\beta$-unstable for $\beta <\alpha$.
Putting altogether we get by definition that $\alpha$ is a critical value.

    Let us now assume that $\alpha$ is a critical value. Suppose that there is no a subsystem $(E',V')\subset (E,V)$ such that $p^{\alpha}_{H,(E',V')}(m)=p^{\alpha}_{H,(E,V)}(m)$. Since $(E,V)$ is $\alpha$-semistable,we have  $p^{\alpha}_{H,(E',V')}(m)<p^{\alpha}_{H,(E,V)}(m)$ for any $(E',V')\subset (E,V)$. Let us now consider $\epsilon \in \mathbb{Q}_{>0}$ such that $\epsilon << 1$. Define $\beta=\alpha+\epsilon$. Then  $p^{\beta-\epsilon}_{H,(E',V')}(m)<p^{\beta-\epsilon}_{H,(E,V)}(m)$ for any $(E',V')\subset (E,V)$ which is equivalent to $$\beta\left(\frac{v'}{n_{E'}}-\frac{v}{n_E}\right)+\frac{P_{H,E'}(m)}{n_{E'}}-\frac{P_{H,E}(m)}{n_E}<\epsilon\left(\frac{v'}{n_{E'}}-\frac{v}{n_E}\right).$$ 
    
    Since $0<\epsilon <<1$, we get $$\beta\left(\frac{k'}{n'}-\frac{k}{n}\right)+\frac{P_{H,E'}(m)}{n_{E'}}-\frac{P_{H,E}(m)}{n_E}<0$$ for any $(E',V')\subset (E,V)$ and thus $(E,V)$ is $\beta$-stable. Analogously, if we take $\beta=\alpha-\epsilon$, we get that $(E,V)$ is $\beta$-stable.

    Putting altogether, $(E,V)$ is $\beta$-stable for any $\beta$ such that $$\alpha-\epsilon\leq\beta\leq\alpha+\epsilon,$$ which contradicts the fact that $\alpha$ is a critical value. 
    
\end{proof}

\begin{Remark} \rm
    Notice that in Proposition \ref{prop_valorcritico} we assume that $\frac{v'}{n_{E'}}\neq\frac{v}{n_{E}}$. If $\frac{v'}{n_{E'}}=\frac{v}{n_{E}}$,  the function $p^\beta_{H,(E',V')}(m)-p^\beta_{H,(E,V)}(m)$  is constant for any $\beta\in \mathbb{Q}[m]_{>0}$, which implies that the notion of stability does not change and thus by definition $\alpha$ is not a critical value.
\end{Remark}

\begin{Notation} \rm
    According to \cite{He}; Theorem 4.2, Given $(n,c_1,c_2,k)$ there is a finite number of critical values
    \[0 = \alpha_0 < \alpha_1 < \cdots < \alpha_s\]
    with $\alpha_i \in \mathbb{Q}[m]_{>0} $ and it in addition $\deg(\alpha_i) < \dim(X)$. 
\end{Notation}

\begin{Definition}
Let $\alpha_{i-1}, \alpha_i, \alpha_{i+1} \in \mathbb{Q}[m]_{>0}$ be consecutive critical values and  let $(E,V)$ be a coherent system.  We will say that $(E,V)$ is $\alpha_i^+-$stable if $(E,V)$ is $\beta-$stable for $\beta \in (\alpha_i, \alpha_{i+1})$.  We will say that $(E,V)$ is $\alpha_i^--$stable if $(E,V)$ is $\beta-$stable for $\beta \in (\alpha_{i-1}, \alpha_{i})$.
\end{Definition}

 \begin{Remark} \label{smallalpha}\rm 
      (1)  If $\alpha_i \in  \mathbb{Q}[m]_{>0}$  is a critical value and $(E,V)$ is a coherent system such that $(E,V)$ is $\alpha_i^+-$stable but $\alpha_i^--$unstable, then by Proposition \ref{prop_valorcritico} $(E,V)$ is strictly $\alpha_i$-semistable.
        Analogously, if $(E,V)$ is $\alpha_i^--$stable but $\alpha_i^+-$unstable, then $(E,V)$ is strictly $\alpha_i$-semistable.
        
        \noindent (2) Let $H$ be an ample divisor on $X$ and $(E,V)$ be a coherent system. The following holds:
        \begin{itemize}
            \item[(a)] If $E$ is Gieseker $H-$stable and $\dim H^0(X,E) \geq k$, then  $(E,V)$ is $0^+-$stable.
            \item[(b)] If $(E,V)$ is $0^+-$stable, then $E$ is Gieseker  $H-$semistable. 
        \end{itemize}

\end{Remark}

\begin{Lemma}
\label{lema6.1}
   Let $(E,V)$ be a coherent system of type $(n_E,c_1,c_2,v)$ and let $(E',V') \subset (E,V)$ be a subsystem of type $(n_{E'},c_1',c'_2,v')$. Then $p^{\alpha}_{H,(E',V')}(m)-p^{\alpha}_{H,(E,V)}(m)$ is a linear function of $\alpha$ which is
   \begin{itemize}
       \item monotonically increasing if $\frac{v'}{n_{E'}}-\frac{v}{n_E}>0$
       \item monotonically decreasing if $\frac{v'}{n_{E'}}-\frac{v}{n_{E}}<0$
       \item constant if $\frac{v'}{n_{E'}}-\frac{v}{n_E}=0$
   \end{itemize}

In particular, if $\alpha_i$ is a critical value and $p^{\alpha_i}_{H,(E',V')}(m)=p^{\alpha_i}_{H,(E,V)}(m)$, then
\begin{itemize}
    \item $(p^{\alpha}_{H,(E',V')}(m)-p^{\alpha}_{H,(E,V)}(m))(\alpha-\alpha_i)>0$, for all $\alpha\neq\alpha_i$ if $\frac{v'}{n_{E'}}-\frac{v}{n_E}>0$,
    \item $(p^{\alpha}_{H,(E',V')}(m)-p^{\alpha}_{H,(E,V)}(m))(\alpha-\alpha_i)<0$, for all $\alpha\neq\alpha_i$ if $\frac{v'}{n_{E'}}-\frac{v}{n_E}<0$,
    \item $p^{\alpha}_{H,(E',V')}(m)-p^{\alpha}_{H,(E,V)}(m)=0$, for all $\alpha$ if $\frac{v'}{n_{E'}}-\frac{v}{n_E}=0$.
\end{itemize}
   
\end{Lemma}

\begin{proof}
    Notice that by definition 
    \begin{equation} \label{falpha}
    f^{\alpha}:=p^{\alpha}_{H,(E',V')}(m)-p^{\alpha}_{H,(E,V)}(m)= \alpha \left(\frac{v'}{n_{E'}}-\frac{v}{n_E}\right)+\frac{P_{H,E'}(m)}{n_{E'}}-\frac{P_{H,E}(m)}{n_E}\end{equation}
    is a linear function on $\alpha$. 

    Let us now see whether this function $f^{\alpha}$ is monotonically decreasing or increasing. Observe that 
    $$f^{\beta}-f^{\alpha}=\beta \left(\frac{v'}{n_{E'}}-\frac{v}{n_E}\right)+\frac{P_{H,E'}(m)}{n_{E'}}-\frac{P_{H,E}(m)}{n_E}-\left[\alpha \left(\frac{v'}{n_{E'}}-\frac{v}{n_E}\right)+\frac{P_{H,E'}(m)}{n_{E'}}-\frac{P_{H,E}(m)}{n_E}\right]$$
    $$=(\beta-\alpha) \left(\frac{v'}{n_{E'}}-\frac{v}{n_E}\right).$$
    
    Then, if  $\frac{v'}{n_{E'}}-\frac{v}{n_E}>0$, we get $f^{\beta}>f^{\alpha}$ for $\beta>\alpha$ which  means that $f^\alpha$ monotonically increases. Analogously, if  $\frac{v'}{n_{E'}}-\frac{v}{n_E}<0$ then $f^\alpha$ monotonically decreases. Finally, if 
    $\frac{v'}{n_{E'}}-\frac{v}{n_E}=0$ then $f^{\beta}=f^{\alpha}$ for any $\beta$ and hence $f^{\alpha}$ is constant.

    Let us now assume that there exists $\alpha_i$ a critical value such that $p^{\alpha_i}_{H,(E',V')}(m)=p^{\alpha_i}_{H,(E,V)}(m)$. This is equivalent to $$\alpha_i \left(\frac{v'}{n_{E'}}-\frac{v}{n_E}\right)=\frac{P_{H,E}(m)}{n_E}-\frac{P_{H,E'}(m)}{n_{E'}}.$$

    Hence, substituting in $(\ref{falpha})$ and multiplying both sides by $(\alpha-\alpha_i)$  we get

 $$(\alpha-\alpha_i)(p^{\alpha}_{H,(E',V')}(m)-p^{\alpha}_{H,(E,V)}(m))= (\alpha-\alpha_i)^2 \left(\frac{v'}{n_{E'}}-\frac{v}{n_E}\right)$$
 and the remaining conclusions easily follow from this equality.
   \end{proof}

\begin{Corollary}
    
\label{monotony}
    Let $\alpha_i \in  \mathbb{Q}[m]_{>0}$ be a critical value and let $(E,V)$ be a coherent system  of type $(n_E,c_1,c_2,v)$ such that $(E,V)$ is strictly $\alpha_i$-semistable.  For any coherent subsystem $(E',V')\subset (E,V)$ of type $(n_{E'},c'_1,c'_2,v')$ such that
    $p^{\alpha_i}_{H,(E',V')}(m)=p^{\alpha_i}_{H,(E,V)}$, the following statements holds:
    \begin{enumerate}
        \item If $(E,V)$ is $\alpha_i^+-$(semi)stable, then $\frac{v'}{n_{E'}} <(\leq) \frac{v}{n_E}$.
        \item If $(E,V)$ is $\alpha_i^--$(semi)stable, then $\frac{v}{n_E} <(\leq) \frac{v'}{n_{E'}}$.
    \end{enumerate}
\end{Corollary}

\begin{proof}
   We will consider the case in which $(E,V)$  is $\alpha_i^+$-stable. The other cases can be drawn using the same arguments. Let us assume that  $\frac{v'}{n_{E'}}\geq\frac{v}{n_E}$. Then, by Lemma \ref{lema6.1}, we get that $p^{\alpha_i}_{H,(E',V')}(m)-p^{\alpha_i}_{H,(E,V)}(m)\geq0$, what contradicts the $\alpha_i^+$-(semi)stability of $(E,V)$. Hence $\frac{v'}{n_{E'}}<\frac{v}{n_E}$.
\end{proof}

\begin{Proposition}
\label{exactseq}
    Let $\alpha_i \in \mathbb{Q}[m]_{>0}$ be a critical value and let $(E,V)$ be a coherent system.
        Assume  that $(E,V)$ is $\alpha_i^+-$stable but $\alpha_i^--$unstable. Then, $(E,V)$ can be written in the following exact sequence of coherent systems
        \[0 \to (E_1,V_1) \to (E,V) \to (E_2,V_2) \to 0\]
        where $(E_i,V_i)$ are coherent systems of type $(n_{E_i},c_1(E_i),c_2(E_i),v_i)$ for $i=1,2$ such that
        \begin{enumerate}
            \item[(a)] $(E_1,V_1), (E_2,V_2)$ are ${\alpha_i}^+-$stable, with $p_{H,(E_1,V_1)}^{\alpha_i^+}(m) < p_{H,(E_2,V_2)}^{\alpha_i^+}(m)$,
        \item[(b)] $(E_1,V_1), (E_2,V_2)$ are $\alpha_i-$semistable, with $p_{H,(E_1,V_1)}^{\alpha_i}(m) = p_{H,(E_2,V_2)}^{\alpha_i}(m)$,
        \item[(c)] $\frac{v_1}{n_{E_1}}$ is a maximum among all proper subsystems $(E_1,V_1)\subset (E,V)$ which satisfy $(b)$,
        \item[(d)] $n_{E_1}$ is a minimum among all subsystems which satisfy $(c)$.
        \end{enumerate}
\end{Proposition}

\begin{proof}
    By Proposition \ref{prop_valorcritico}, the coherent system $(E,V)$ is strictly $\alpha_i-$semistable. In particular, there exists  a proper subsystem $(E_1,V_1)\subsetneq(E,V)$ such that $$p_{H,(E,V)}^{\alpha_i}(m) = p_{H,(E_1,V_1)}^{\alpha_i}(m).$$

     Consider the non-empty set
    \[\mathcal{F}_1=\{(E_1,V_1) \subset (E,V) \thinspace|\thinspace p_{H,(E,V)}^\alpha(m) = p_{H,(E_1,V_1)}^\alpha(m) \}.\]
    Notice that, for any $(E_1,V_1) \in \mathcal{F}_1$, $n_{E_1} < n_E$ and $V_1 = V\cap H^0(X,E_1)$. In fact, if  $V_1$ is strictly contained in $W:= V\cap H^0(X,E_1)$, we can consider the subsystem $(E_1,W) \subset (E,V)$ which satisfies $p^{\alpha_i}_{H,(E_1,W)}(m)>p^{\alpha_i}_{H,(E,V)}(m)$, but this contradicts the $\alpha_i$-semistability of $(E,V)$. Hence,  $V_1 = V\cap H^0(X,E_1)$.  
    
    Since $(E,V)$ is $\alpha_i^+-$stable,   by Corollary \ref{monotony}, for any $(E_1,V_1) \in \mathcal{F}_1$ we have $\frac{v_1}{n_{E_1}}< \frac{v}{n_E}$. Thus, since  
    the values for $\frac{v_1}{n_{E_1}}$ are limited by the constraints $0<n_{E_1}<n_E$ and $0\leq v_1\leq v$, we can define 
       \[\lambda_0=\max\left\{\frac{v_1}{n_{E_1}} \thinspace|\thinspace (E_1,V_1) \in \mathcal{F}_1\right\}\]
    and set 
    \[{\mathcal{F}_2}=\left\{(E_1,V_1) \in \mathcal{F}_1 \thinspace|\thinspace \frac{v_1}{n_{E_1}}=\lambda_0\right\}.\]

Let $(E_1,V_1)$ be a coherent subsystem in $\mathcal{F}_2$. Since $V_1=V\cap H^0(X,E_1)$, we can consider the exact sequence 
\begin{equation}
    \label{seq_cs}
    0\rightarrow (E_1,V_1)\rightarrow (E,V)\rightarrow (E_2,V_2)\rightarrow0
\end{equation}
for some coherent system $(E_2,V_2)$ with $E_2$ torsion-free (see Proposition \ref{torsionfree}). Moreover, since $p_{H,(E,V)}^{\alpha_i}(m) = p_{H,(E_1,V_1)}^{\alpha_i}(m)$ from $(\ref{seq_cs})$ we have 
 \begin{equation} \label{eq1} p_{H,(E,V)}^{\alpha_i}(m) = p_{H,(E_1,V_1)}^{\alpha_i}(m)=p_{H,(E_2,V_2)}^{\alpha_i}(m). \end{equation} Therefore since $(E,V)$ is $\alpha_i$-semistable,  $(E_1,V_1)$ and $(E_2,V_2)$ are both $\alpha_i$-semistable.

Let us now prove  that $(E_2,V_2)$ is $\alpha^+_i$-stable. Suppose it is not.
In this case, by Proposition \ref{prop_valorcritico}, $(E_2,V_2)$ must be strictly $\alpha_i$-semistable and in particular  there exists a proper subsystem $(E'_2,V'_2)\subset (E_2,V_2)$  of type $(n_{E'_2},c_1(E'_2),c_2(E'_2),v'_2)$ such that 
\begin{equation} \label{eq2}
    p_{H,(E'_2,V'_2)}^{\alpha_i}(m) = p_{H,(E_2,V_2)}^{\alpha_i}(m). 
\end{equation} 
 Moreover, by Corollary \ref{monotony}, $\frac{v'_2}{n_{E'_2}} \geq \frac{v_2}{n_{E_2}}$.

 Therefore, there exists $(E_2',V_2') \subset (E_2,V_2)$ such that
 \begin{equation}
 \label{eq3}
 p_{H,(E'_2,V'_2)}^{\alpha_i}(m) = p_{H,(E_2,V_2)}^{\alpha_i}(m)  \quad \mbox{and} \quad  \frac{v'_2}{n_{E'_2}} \geq \frac{v_2}{n_{E_2}}.
 \end{equation}
  
    Consider now the subsystem $(E',V')\subset (E,V)$ defined by the pull-back diagram
\begin{equation}
    \label{seq_cs2}
0\rightarrow(E_1,V_1)\rightarrow(E',V')\rightarrow(E'_2,V'_2)\rightarrow 0.
\end{equation}
From $(\ref{eq1})$ and $(\ref{eq3})$  we have 
\[p_{H,(E'_2,V'_2)}^{\alpha_i}(m) = p_{H,(E_2,V_2)}^{\alpha_i}(m)=p_{H,(E_1,V_1)}^{\alpha_i}(m).\]
Therefore, using $(\ref{seq_cs2})$ we obtain
$$p_{H,(E',V')}^{\alpha_i}(m) = p_{H,(E_1,V_1)}^{\alpha_i}(m).$$ Finally,  since $p_{H,(E,V)}^{\alpha_i}(m) = p_{H,(E_1,V_1)}^{\alpha_i}(m)$, we get that $p_{H,(E',V')}^{\alpha_i}(m) = p_{H,(E,V)}^{\alpha_i}(m)$.

Hence $(E',V')\in \mathcal{F}_1$ and since $(E_1,V_1)\in \mathcal{F}_2$, by definition of $\mathcal{F}_2$ and the exact sequence $(\ref{seq_cs2})$ we get
$$\frac{v_1+v_2'}{n_{E_1}+n_{E'_2}}=\frac{v'}{n_{E'}}\leq\frac{v_1}{n_{E_1}}$$ which is equivalent to $\frac{v'_2}{n_{E'_2}}\leq \frac{v_1}{n_{E_1}}$. On the other hand, we have seen that $\frac{v_1}{n_{E_1}}<\frac{v_2}{n_{E_2}}$. Putting altogether we get
$$\frac{v'_2}{n_{E'_2}}\leq \frac{v_1}{n_{E_1}} <\frac{v_2}{n_{E_2}},$$
which contradicts $(\ref{eq3})$. Therefore, $(E_2,V_2)$ is $\alpha^+_i$-stable. 

Let us now consider $(E_1,V_1)\in \mathcal{F}_2$ with minimum rank in $\mathcal{F}_2$. Let us prove that $(E_1,V_1)$ is $\alpha^+_i$-stable. If not, by the same argument than before, there exists a proper subsystem $(E'_1,V'_1)\subset (E_1,V_1)$ such that 
    \[ p_{H,(E'_1,V'_1)}^{\alpha_i}(m) = p_{H,(E_1,V_1)}^{\alpha_i}(m) \quad \mbox{and} \quad  \frac{v'_1}{n_{E'_1}}\geq\frac{v_1}{n_{E_1}}. \]
   Since $v'_1\leq v_1$ and $\frac{v'_1}{n_{E'_1}}\geq\frac{v_1}{n_{E_1}}$, we get that $n_{E'_1}<n_{E_1}$, which contradicts the minimality of $n_{E_1}$. Therefore, $(E_1,V_1)$ is $\alpha_i^+$-stable.

   Finally, notice that since $(E,V)$ is $\alpha^+_i$-stable, by (\ref{seq_cs}) we must have 
   $$p_{H,(E_1,V_1)}^{\alpha_i^+}(m) < p_{H,(E,V)}^{\alpha_i^+}(m)<p_{H,(E_2,V_2)}^{\alpha_i^+}(m).$$
\end{proof}

\section{Clifford's Theorem for Coherent Systems}

In  this section, pushing forward the ideas contained in \cite{Lange-Newstead} for the cases of curves,  we establish the Clifford's Theorem for coherent systems on surfaces. One of the ingredients that we use is  the following reformulation of a generalization
of Clifford’s Theorem for vector bundles on surfaces  \cite[Theorem 4.1]{Costa}.

\begin{Theorem}\label{Costa-Miro-Roig}
    Let $X$ be a smooth algebraic surface,  $H$ an ample divisor on $X$ such that $K_X\cdot H\leq0$ and let $E$ be a rank $n\geq1$ semistable  vector bundle on $X$. Let $a$ be an integer such that
 \[\deg(X)\cdot\max\{\frac{n^2-1}{4},1\} < \frac{\binom{a+2}{2}-a-1}{a}.\]   If  $$0\leq \frac{c_1(E)H}{n}<aH^2+K_XH,$$ then \[h^0(E)\leq n+a\frac{c_1(E)H}{2}. \]
\end{Theorem}

\begin{proof} It is enough to observe that the proof of  \cite[Theorem 4.1]{Costa} also includes the case $n=1$. 
\end{proof}

 \begin{Corollary}
 \label{cor_clif}
    Let $X$ be a smooth algebraic surface, $H$ an ample divisor on $X$ such that $K_X\cdot H\leq0$ and let $E$ be a rank $n\geq1$ semistable  torsion-free sheaf on $X$. Let $a$ be an integer such that
 \[\deg(X)\cdot\max\{\frac{n^2-1}{4},1\} < \frac{\binom{a+2}{2}-a-1}{a}.\]   If  $$0\leq \frac{c_1(E)H}{n}<aH^2+K_XH,$$ then \[h^0(E)\leq n+a\frac{c_1(E)H}{2}. \]
\end{Corollary}

\begin{proof}
    For any coherent sheaf $E$ on $X$ there is a natural homomorphism $\mu:E \to  (E^*)^*$ of $E$ into its double dual $(E^*)^*$. This morphism is injective if and only if $E$ is torsion-free. Moreover, since $X$ is a surface, $(E^*)^*$ is locally free. Therefore, $$h^0(E)\leq h^0(E^*)^*\leq n+\frac{ac_1(E)H}{2},$$ where the last inequality follows from Theorem \ref{Costa-Miro-Roig}.   
\end{proof}

Now we are ready to prove the main result of this section

\begin{Theorem} \label{Cliffordtorsionfree}
Let $X$ be a smooth projective surface and $H$ an ample divisor on $X$ such that $K_X\cdot H\leq0$.
    Let $(E,V)$ be a coherent system on $X$ of type $(n_E,c_1,c_2,v)$ with $v>0$ which is $\alpha-$semistable for some $\alpha \in \mathbb{Q}[m]_{>0}$. Let $a$ be an integer such that
 \[\deg(X)\cdot\max\{\frac{n^2_E-1}{4},1\}<\frac{\binom{a+2}{2}-a-1}{a}.\]  If $$0\leq \frac{c_1 H}{n_E}<aH^2+K_X H,$$  
then 
    \[v \leq n_E+a\cdot \frac{c_1 H}{2}.\]
   
\end{Theorem}

\begin{proof}
The proof is by induction on $n=n_E$.  For the case $n=1$, let $(L\otimes I_Z, V)$ be a coherent system of type $(1,L,|Z|,v)$ which is $\alpha-$semistable for some $\alpha \in \mathbb{Q}[m]_{>0}$. 
Let us take $a$ an integer such that $$\deg(X) < \frac{\binom{a+2}{2}-a-1}{a}.$$ 
Since $L\otimes I_Z$ is a rank $1$ torsion-free-sheaf, we have that  $(L\otimes I_Z, V)$ is $H$-semistable.  Therefore, from 
Corollary \ref{cor_clif}, it follows that if $0\leq c_1(E)H<aH^2+K_XH$ we get

    \[v \leq h^0(L\otimes I_Z) \leq 1+a\frac{LH}{2}. \]

Assume $n\geq 2$  and the theorem is proved for coherent systems of rank less than $n$.

Let $(E,V)$ be a coherent system of type $(n,c_1,c_2,v)$ which is $\alpha-$semistable for some $\alpha \in \mathbb{Q}[m]_{>0}$.
Let $a$ be an integer such that $$\deg(X)\cdot\max\{\frac{n^2-1}{4},1\}<\frac{\binom{a+2}{2}-a-1}{a}$$ 
and  assume that $0\leq \frac{c_1(E)\cdot H}{n}<aH^2+K_X H.$

First, we consider the case in which $(E,V)$ is $0^{+}$-stable. This implies that $E$ is $H$-semistable (see Remark \ref{smallalpha}) and hence by Corollary \ref{cor_clif}  we have
    \[v\leq h^0(E)\leq n+a\frac{c_1(E)H}{2}. \]

Let us now assume that $(E,V)$ is not $0^{+}$-stable. Then, by Proposition \ref{prop_valorcritico}, $(E,V)$ is strictly $\alpha$-semistable for some $\alpha>0$. Moreover, by Proposition \ref{exactseq}, $(E,V)$ sits on the exact sequence
 \[0 \to (E_1,V_1) \to (E,V) \to (E_2,V_2) \to 0,\]
        where
     $(E_1,V_1)$ and $(E_2,V_2)$ are $\alpha-$semistable coherent systems of type $(n_{E_i}, c_1(E_i), c_2(E_i),v_i)$ with $p_{H,(E_i,V_i)}^\alpha(m) = p_{H,(E,V)}^\alpha(m)$, for $i=1,2$.

\noindent \textbf{Claim 1:} $\frac{v_1}{n_{E_1}}\leq \frac{v_2}{n_{E_2}}$ and $\frac{c_1(E_1)\cdot H}{n_{E_1}}\geq \frac{c_1(E_2)\cdot H}{n_{E_2}} $.

\noindent \textbf{Proof of Claim 1:}
Since $p_{H,(E_1,V_1)}^\alpha(m) = p_{H,(E,V)}^\alpha(m)$,   by Corollary \ref{monotony} we get 
\begin{equation}
\frac{v_1}{n_{E_1}}\leq \frac{v}{n}.
    \label{cota1}
\end{equation}
Since $v_1+v_2=v$ and $n_{E_1}+n_{E_2}=n$, this implies   
\begin{equation}
    \label{cota2}
\frac{v_1}{n_{E_1}}\leq \frac{v_2}{n_{E_2}}.
\end{equation}

Since $$\frac{v_1}{n_{E_1}}\leq \frac{v_2}{n_{E_2}}\thinspace \thinspace \mbox{and }\thinspace\thinspace p_{H,(E_1,V_1)}^\alpha(m) = p_{H,(E_2,V_2)}^\alpha(m),$$ we have $$\frac{P_{H,E_1}(m)}{n_{E_1}}\geq\frac{P_{H,E_2}(m)}{n_{E_2}},$$  which implies  $$\frac{c_1(E_1)H}{n_{E_1}}\geq \frac{c_1(E_2) H}{n_{E_2}}. $$

\noindent\textbf{Claim 2: $v_2\leq r_2+a\frac{c_1(E_2)\cdot H}{2}.$}  

\noindent \textbf{Proof of Claim 2:}

By (\ref{cota1}) and the fact that $v=v_1+v_2$, we get 
\begin{equation}
    \label{cota3}
    \frac{v_2}{n_{E_2}}\geq \frac{v}{n}.
\end{equation}

Since $p_{H,(E_2,V_2)}^\alpha(m) = p_{H,(E,V)}^\alpha(m),$ the bound (\ref{cota3}) implies that $$\frac{P_{H,E}(m)}{r}\geq\frac{P_{H,E_2}(m)}{r_2}$$ and hence $$\frac{c_1(E) H}{r}\geq \frac{c_1(E_2) H}{r_2}. $$

Moreover, since $(E_2,V_2)$ is $\alpha$-semistable, we have that $c_1(E_2) H\geq0$.

Putting altogether,
$$0\leq \frac{c_1(E_2) H}{n_{E_2}} \leq \frac{c_1(E)H}{r} <aH^2+K_X H,$$

where $$\deg(X)\cdot\max\{\frac{{n^2_{E_2}}-1}{4},1\} < \deg(X)\cdot\max\{\frac{n^2-1}{4},1\}< \frac{\binom{a+2}{2}-a-1}{a}.$$

Thus we are under assumption of induction hypothesis and hence $$v_2\leq n_{E_2}+a\frac{c_1(E_2) H}{2}.$$

By (\ref{cota1}), this implies that $v_1\leq n_{E_1}+a\frac{c_1(E_1) H}{2}$.

Putting altogether,
$$v=v_1+v_2\leq n_{E_1}+n_{E_2}+a\cdot\frac{(c_1(E_1)+c_1(E_2)) H}{2}=n+a\cdot\frac{c_1(E) H}{2}.$$
\end{proof}

\begin{Corollary}
    \label{cor_Clifford_CS}
Let $X$ be a smooth projective surface and $H$ an ample divisor on $X$ such that $K_X\cdot H\leq0$. 
    Let $(E,V)$ be a coherent system on $X$ of type $(n_E,c_1,c_2,v)$ with $v>0$ which is $\alpha-$semistable for some $\alpha \in \mathbb{Q}[m]_{>0}$. Set $a=\lceil \frac{(n^2_E-1)H^2}{2} \rceil$ or $a=2H^2$ if $n_E=1,2$.   If $$0\leq \frac{c_1 H}{n_E}<aH^2+K_X H,$$  
then 
    \[v \leq n_E+a\cdot \frac{c_1 H}{2}.\]
   
\end{Corollary} 

\begin{proof}
    It  follows from Theorem \ref{Cliffordtorsionfree} since $a$ is the  integer satisfying  $$\deg(X)\cdot\max\{\frac{n^2_E-1}{4},1\} < \frac{\binom{a+2}{2}-a-1}{a}.$$
\end{proof}

\begin{Remark} \rm
The bound in  Corollary \ref{cor_Clifford_CS} is not necessarily sharp, at least for $0^+$-semistable coherent systems as it we can shown be means of the next example. 
\end{Remark}

\begin{Example} \rm  Let $X$ be a ruled surface over a nonsingular curve $C$ of genus $g\geq0$ and with invariant $e>0$. Denote by $C_0$ and $f$ the generators of the Picard group of $X$ such that any divisor $D$ is numerically equivalent to $aC_0+bf$ for some $a,b \in \ZZ$, where $C_0$ represents a section of $X \to C$ and $f$ a class of a fiber. Take $H \equiv C_0+\beta f$ with $\beta >>0$ an ample divisor on $X$ and  $c_2>>0$ an integer. In  \cite{Costa2}, the authors proved that for all rank 2 $H$-stable vector bundles $E$ on $X$ with $c_1(E)=f$ and $c_2(E)=c_2$, $$h^0(E) \leq 2.$$ 
Moreover, if $E$ is $H$-stable then  the pair $(E,V)$, with $V \subset H^0(E)$ of dimension 1, is $0^+$-semistable and satisfies
$$h^0(E) \leq 2 $$
which is clearly a better bound than in Corollary \ref{cor_Clifford_CS}. 
\end{Example}

\end{document}